\documentclass[12pt]{article}
\usepackage{amsmath,amsfonts,amssymb,amsthm,amscd}
\title{On minimal flows,  definably amenable groups, and  $o$-minimality}
\date{\today}
\author{Anand Pillay\thanks{Partially supported by NSF grant DMS-1360702}\\University of Notre Dame\and Ningyuan Yao\thanks{Supported by a grant from the Chinese government and by the University of Notre Dame}\\Sun Yat-Sen University}
\newtheorem{Theorem}{Theorem}[section]
\newtheorem{Proposition}[Theorem]{Proposition}
\newtheorem{Definition}[Theorem]{Definition} 
\newtheorem{Remark}[Theorem]{Remark}
\newtheorem{Lemma}[Theorem]{Lemma}
\newtheorem{Corollary}[Theorem]{Corollary}
\newtheorem{Fact}[Theorem]{Fact}

\newtheorem{Question}[Theorem]{Question}

\newcommand{\R}{\mathbb R}

\newcommand{\N}{\mathbb N}

\begin{document}
\maketitle

\begin{abstract}  We study  definably amenable groups in $NIP$ theories, focusing  on the problem  raised in \cite{Newelski1} of  whether weak generic types coincide with almost periodic types, equivalently whether the union of minimal subflows of a suitable type space is closed.  We give fairly definitive results in the $o$-minimal context, including a counterexample.

\end{abstract}
\section{Introduction and preliminaries. } Given a group $G$ definable over (or in) a structure $M$ we have the action, by homeomorphisms,  of $G$ on the space $S_{G}(M)$ of complete types over $M$ which concentrate on $G$. Various invariants of such an action, including minimal  $G$-subflows of $S_{G}(M)$, are suggested by topological dynamics, and in the case when $Th(M)$ is stable, coincide with invariants at the heart of stable group theory such as the space of generic types.  When $Th(M)$  is not necessarily stable, these invariants are useful for generalizing stable group theory. This theme has been pursued in several recent papers, sometimes under the assumption that $T$ has $NIP$.  
The  class of definably amenable groups is a reasonable choice for the class of ``stable-like" groups, in this $NIP$ environment. For example, definable amenability can be characterized by the existence of  ``generic types"  in the sense of forking. 

As discussed in the original   paper on the topic, namely  \cite{Newelski1} , {\em almost periodic} types are among  the first ``new" objects suggested by the dynamics point of view, where $p\in S_{G}(M)$ is by definition almost periodic  if the closure of the $G$-orbit of $p$ is a minimal closed $G$-invariant subset of $S_{G}(M)$.   On the other hand {\em weakly generic} formulas and types were introduced in \cite{Newelski-Petrykowski} as a substitute for generic formulas and types as the latter may not always exist. Briefly, a definable set (or formula) $X\subseteq G$  is weakly generic  if $X\cup Y$ is generic for some nongeneric definable $Y$. Where a definable set is  generic if finitely many translates cover the whole group.

Among  the nice observations in \cite{Newelski1} was that the class of weak generic types in $S_{G}(M)$ is precisely the closure of the class of almost periodic types. An example was given where the two classes differ and the problem was explicitly raised (Problem 5.4 of \cite{Newelski1}) of finding an $o$-minimal or even just $NIP$ example.  This is what we address  in the current paper.    Newelski's question had nothing to do with definable amenability in itself, but there is no harm in looking within the class of definably amenable groups  for a counterexample. In fact  Newelski's question is essentially restated in  \cite{Chernikov-Simon} (Question 3.35)  in the special case of definably amenable groups in $NIP$ theories.

The advantage of working in the $o$-minimal case is that we have a good understanding of definably amenable groups; they are precisely definable groups $G$ fitting into a (definable) short exact sequence $$1\to H\to G \to C\to 1$$
 where $H$ is torsion-free and $C$ is ``definably compact" \cite{CP}. 

 There are at least two contexts for studying the topological dynamical invariants. What we call the {\em global context}  is where $M$ is a saturated model of $T$.  What we call the {\em local context}  is where $M$ is any model of $T$, and we pass to the Shelah expansion $M_{0} = M^{ext}$ of $M$ by externally definable sets and consider instead the action  of $G$ on $S_{G}(M_{0})$. Each context has its advantages. For example in the local context almost all the machinery of topological dynamics becomes available (such as the semigroup structure on the type space), and moreover certain results will transfer to $S_{G}(M)$.    The current paper can also be seen as beginning a ``fine" description of the minimal subflows of the various type spaces in the definably amenable $o$-minimal situation. 

We begin with the positive results; namely equality of weak generics and almost periodics in some cases. 

\begin{Theorem}  (Global case). Suppose $T$ is $o$-minimal, $G$ a definable, definably amenable group, defined over a saturated model $M$ of $T$. Suppose that either $G= H$ or $dim(H)\leq 1$. 
Then the set of weakly generic types over $M$ coincides the with the set of almost periodic types over $M$. 

\end{Theorem}

\begin{Theorem} (Local case)
Suppose $T$ is $o$-minimal and $G$ is a definable, definably amenable group, definable over a model $M$ of $T$. Suppose moreover that $G = C\times H$ and either $G = H$ or $dim(H)\leq 1$. Let $M_{0} = M^{ext}$. Then working in $S_{G}(M_{0})$ the set of weakly generic types coincides with the set of almost periodic types.  Likewise for $S_{G}(M)$, weak generic types coincide with almost periodic types.
\end{Theorem}

Note that if the underlying set of the $o$-minimal structure $M$ is $\R$ then $M^{ext}$  coincides with $M$ (by \cite{Marker-Steinhorn}). Bearing in mind the above positive results, the next result gives a ``minimal counterexample". 

\begin{Theorem} Let $M =\R$ be the standard model of $RCF$ (i.e. the ordered field of real numbers). Let $G = S^{1}\times (\R,+)^{2}$. Then working either in $S_{G}(\R)$, or in $S_{G}({\bar M})$ for $\bar M$ a saturated model, the set of weakly generic types properly contains the set of almost periodic types. 
\end{Theorem}

\vspace{5mm}
\noindent

Finally, working again in the $o$-minimal, definably amenable setting,  we prove that the restriction of a minimal flow to a smaller model is also a minimal flow. This question was raised in an early draft of \cite{Chernikov-Simon}. Pierre Simon recently told us that he had solved it in the general $NIP$ context. Nevertheless we include our proof in the $o$-minimal case, as it is an easy consequence of our set-up. 

\begin{Theorem} Suppose $T$ is $o$-minimal, and $G$ is a definably amenabnle group definable over $M$. Let $M\prec N$, let $\pi:S_{G}(N) \to S_{G}(M)$ be the canonical restriction map. Let ${\cal M}$ be a minimal $G(N)$-subflow of $S_{G}(N)$, then $\pi({\cal M})$ is a minimal $G(M)$-subflow of $S_{G}(M)$. 

\end{Theorem}

\vspace{2mm}
\noindent
For the rest of this introductory section we give precise definitions and background. In fact we will give a bit more than is strictly needed for proving the main results of the paper. So this introductory section may be considered as a partial survey of the state-of-the-art in the subject.

\vspace{5mm}
\noindent
$T$ will denote a complete theory, $M, N$
.. models of $T$ and ${\bar M}$ a saturated model. If  $X$ is a definable set (in particular a definable group $G$), defined over the model $M$ then we write $S_{X}(M)$ for the space of complete types concentrating on $X$.  We identify $X$, $G$,.. with points in $\bar M$.  We use freely basic notions of model theory such as definable type, heir, coheir,.... The book \cite{Poizat} is a possible source. 

\subsection{Topological dynamics}
Topological dynamics is the study of topological (often discrete) groups via their actions on compact spaces. References for the basic theory are \cite{Auslander} and \cite{Glasner}. In \cite{GPPII}, the first author together with Gismatullin and Penazzi, influenced by \cite{Newelski1},  tried to {\em adapt} this theory to the category of definable (in some given structure $M$) groups, in place of discrete groups, in the sense of identifying universal minimal ``definable" flows etc.    Although relevant to the current paper, knowledge of this ``theory" from \cite{GPPII} is not strictly required. 

We start with some basic definitions, at a suitable level of generality. Let $G$ be a group, $X$ a compact Hausdorff space and $G\times X \to X$ an action of $G$ on $X$ by homeomorphisms.  So the action $G\times X \to X$ is continuous when $G$ is equipped with the discrete topology. 
We will be assuming that there is a dense orbit, although it is not always required.   We call $X$ a $G$-flow, and by a subflow we mean a closed $G$-invariant subset of $X$. Minimal subflows of $X$ exist (as an intersection of subflows is also a subflow), and a key notion is that of an almost periodic point:
\begin{Definition} $x\in X$ is said to be {\em almost periodic} if the closure $cl(G.x)$ of the orbit $G.x$ of $x$ under $G$ is a minimal subflow of $X$.  Equivalently $x\in X$ is almost periodic if $x$ is in some minimal subflow of $X$. 
\end{Definition} 

The notion ``generic" is well-known in topological dynamics where it goes under the name syndetic, whereas the notion weak generic was introduced in \cite{Newelski1} 

\begin{Definition} (i) A subset $Y$ of $X$ is {\em generic}  if finitely many $G$-translates cover $X$.
\newline
(ii) A subset $Y$ of $X$ is {\em weakly generic} if for some finite union $Z$ of $G$-translates of $Y$, $X\setminus Z$ is non generic.
\newline
(iii) A point $x\in X$ is {\em generic} if every open neighbourhood $U$ of $x$ is generic.
\newline
(iv) Likewise a point $x\in X$ is weak(ly) generic if every open neighbourhood of $x$ is weakly generic. 
\end{Definition} 

The following summarizes the relationship between these notions. It is taken from \cite{Newelski1}, although (ii) is well-known. 

\begin{Lemma} (i) The set of weak generic points of $X$ is precisely the closure of the set of almost periodic points.
\newline
(ii) If there is a generic point in $X$ then there is a unique minimal subflow of $X$ which moreover coincides with the set of generic points. So also generic points, almost periodic pioints, and weak generic points coincide. 
\end{Lemma} 

\subsection{Model-theoretic context} 
We consider a complete theory $T$, model $M$ of $T$, group $G$ definable over $M$ and the action of $G(M)$ on the (profinite) type-space space $S_{G}(M)$, on the left say:  $gp = tp(ga/M)$ where $a$ realizes $p$.  This is an action by homeomorphisms and there is a dense orbit, namely $G(M)$ itself (considered as a subset of $S_{G}(M)$). Hence the definitions and results of  the previous subsection apply.  The Boolean algebra of definable subsets of $G(M)$ is naturally isomorphic to the Boolean algebra of clopen subsets of $S_{G}(M)$. For $Y$ a definable subset of $G(M)$, let $[Y]$ be the corresponding clopen.

\begin{Remark} Let $Y$ be a definable subset of $G(M)$. Then
\newline
(i) $[Y]$ is generic in the sense of Definition 1.6 if $Y$ is left generic in the usual model-theoretic sense, namely finitely many left $G(M)$-translates of $Y$ cover $G(M)$.
\newline 
(ii) $[Y]$ is weakly generic in the sense of Definition 1.6 iff there is a finite union $Z$ of left $G(M)$-translates of $Y$ such that $G(M)\setminus Z$ (a definable subset of $G(M)$) is not left generic. 
\end{Remark} 

So we take the right hand side in (ii) to be the {\em definition} of a definable subset $Y$ of $G(M)$ being weakly (left) generic, relative to $M$.  An equivalent definition is: for some definable non (left) generic subset $Y'$ of $G(M)$, $Y\cup Y'$ is (left) generic.  Let us remark that if $Y = Y({\bar M})$ is a definable subset of $G = G({\bar M})$, defined over $M$, and $Y(M)$ is weakly generic (with respect to $M$) then $Y({\bar M})$ is weakly generic with respect to ${\bar M}$. But the converse need not be the case. Anyway for $Y$ a definable subset of $G$ we will say that $Y$ is {\em globally weakly generic} if it is weakly generic with respect to ${\bar M}$.

\begin{Remark} (i) $p\in S_{G}(M)$ is generic  in the sense of Definition 1.6 iff every definable set (formula)  in $p$ is (left) generic.
\newline
(ii) Likewise $p\in S_{G}(M)$ is weak generic in the sense of 1.6 iff every definable set (formula) in $p$ is weakly (left)  generic (with respect to $M$). 
\end{Remark}

In  section 3 of \cite{Newelski1} two examples, coming essentially from topological dynamics, are given where weak generic types do not coincide with almost periodic types.  

\vspace{5mm}
\noindent
By an {\em externally definable} subset of $G(M)$ we mean the trace on $G(M)$ (i.e. intersection with $G(M)$) of a definable, with parameters, subset of $G = G({\bar M})$. We will denote the Boolean algebra of externally definable subsets of $G(M)$ by $B_{ext}(G(M))$. Let  $S_{G,ext}(M)$ denote the Stone space of $B_{ext}(G(M))$. Let $M^{ext}$ be the 
expansion of $M$ obtained by adding predicates for all externally definable sets in $M$. Then $S_{G,ext}(M)$ is simply the collection of complete quantifier-free types over $M^{ext}$. 
When $T$ has $NIP$ then $Th(M^{ext})$ has quantifier elimination whereby $S_{G,ext}(M)$ coincides with the space $S_{G}(M^{ext})$ of 
complete types over $M^{ext}$ concentrating on $G$. But we will not be assuming $NIP$ for now.  Now always ($NIP$ or no $NIP$) the space $S_{G,ext}(M)$ is naturally homeomorphic 
to the space $S_{G,M}({\bar M})$ of global complete types concentrating on $G$ which are finitely satisfiable in $M$: If $p(x)\in S_{G,ext}(M)$, $X$ is a definable (with parameters) 
subset of $G({\bar M})$ and $X\cap G(M)\in p$ then put $X\in p'$,  to obtain $p'\in S_{G,M}({\bar M})$.  Conversely if $p'\in S_{G,M}({\bar M})$ let $p$ be the set of $X\cap G(M)$ for $X\in p'$.

\vspace{2mm}
\noindent
  In any case, whatever the point of view,  we have the natural action of $G(M)$ on the space $S_{G, ext}(M)$, also by homeomorphisms, and Remarks 1.8 and 1.9, as well as appropriate definitions of (weak) genericity for externally definable subsets of $G(M)$,  remain valid in this context. 

\vspace{2mm}
\noindent
\begin{Definition} (i) For $p\in S_{G}(M)$, and $X$ a definable subset of $G(M)$, let $d_{p,M}(X) = \{g\in G(M): X\in gp\}$.
\newline
(ii) For $p\in S_{G,ext}(M)$ and $X$ an externally definable subset of $G(M)$, let $d_{p,M}(X)  = \{g\in G(M): X\in gp\}$
\end{Definition}

\begin{Remark} (i) Let $p\in S_{G}(M)$ and let $X$ be a definable subset of $G(M)$, defined by formula $\phi(x)$ (over $M$).  Realize $p$ by $a\in G({\bar M}$. Then 
$d_{p,M}(X)  = \{g\in G(M): {\bar M}\models \phi(ga)\}$, hence is externally definable.
\newline
(ii) Likewise if $p\in S_{G,ext}(M)$ and $X$ is an externally definable subset of $G(M)$, then $d_{p,M}(X)$ is an externally definable subset of $G(M)$.
\newline
(iii) Suppose $X$ is a definable subset of $G(M)$, $p\in S_{G,ext}(M)$ and $p_{0}\in S_{G}(M)$ is the restriction of $p$ to definable subsets of $G(M)$. Then $d_{p,M}(X) = d_{p_{0},M}(X)$. 
\newline
(iv) Let $M_{1}\prec M_{2}$ be models of $T$, let $q\in S_{G}(M_{2})$, and $p = q|M_{1}$. Let $\phi(x)$ be a formula over $M_{1}$. Then $d_{p,M_{1}}(\phi(M_{1}) = d_{q,M_{2}}(\phi(M_{2})) \cap G(M_{1})$. 
\newline
(v) Let $p\in S_{G}(M)$ or $S_{G,ext}(M)$. Then $p$ is almost periodic iff for each $X\in p$, $d_{p,M}(X)$ is left generic (finitely many left $G(M)$-translates cover $G(M)$). 

\end{Remark}
\begin{proof} (i) to (iv) are clear. And (v) is contained in the proof of Fact 4.3 of \cite{CPS}. 
\end{proof}

We now recall the semigroup stucture $*$ on $S_{G,ext}(M)$.  There are various equivalent constructions. Here is one of them.  Let $p,q\in S_{G, ext}(M)$ identified as above with global types $p'$, $q'$  in $S_{G,M}({\bar M})$. Let $b$ realize $q'$ (in a saturated  elementary extension of ${\bar M}$) and let $b$ realize the unique extension of $p'$ over ${\bar M},b$ which is finitely satisfiable in $M$. Then $tp(a\cdot b/{\bar M})$ is finitely satisfiable in $M$, so is in $S_{G,M}({\bar M})$ and $p*q$ is defined to be the corresponding type in $S_{G,ext}(M)$. 

\begin{Lemma} Let $p,q\in S_{G,ext}(M)$, and let $X$ be an externally definable subset of $G(M)$. Then $X\in p*q$ iff $d_{q,M}(X) \in p$. 
\end{Lemma} 
\begin{proof} Note first that by Remark 1.11  $d_{q,M}(X)$ is an externally definable subset of $G(M)$ so the right hand side of the conclusion makes sense. 
Suppose first that $X\in p*q$.  Let $p',q'\in S_{G,M}({\bar M})$ correspond to $p,q$ respectively as discussed above.  
Let $a'$, $b'$ realize $p',q'$ such that $tp(a'/b',{\bar M})$ is the unique extension of $p'$ which is finitely satisfiable in $M$.  And let $X = Y\cap G(M)$ for $Y$ a definable subset of $G({\bar M})$.  So $Y\in tp(a\cdot b/{\bar M})$.  Namely $a\cdot b\in Y$ so $b\in a^{-1}\cdot Y$. We claim that $\{a'\in G(M):a'^{-1}X\in q\}$ is in $p$. Otherwise the negation is in $p$ and we get a contradiction.The converse is similar. 
\end{proof} 

\subsection{ $NIP$ and definable amenability}
Not all of this section is required for the main results but it puts these results in context and also relates notions such as weak generic types to notions from stability theory.  We also solve here positively the ``weak generic = almost periodic" question in some easy cases. 

We first recall the stability-theoretic notion of dividing: A type $p(x)\in S(B)$ divides over a set $A\subseteq B$ if there is a formula $\phi(x,b)\in p$ and infinite $A$-indiscernible sequence $(b= b_{0},b_{1}, b_{2},....)$ such that $\{\phi(x,b_{i}):i<\omega\}$ is inconsistemt. 

Recall also the notion of a Keisler measure over $M$ on $X$ where $X$ is a definable set, definable over $M$: it is precisely a finitely additive probability measure on the Boolean algebra of definable, over $M$, subsets of $X$.  When we choose $M = {\bar M}$ we call it a global Keisler measure on $X$. 
The definable group $G$ is said to be {\em definably amenable} if there is a global (left) $G$-invariant  Keisler measure on $G$.    By \cite{NIPI} this is equivalent to the existence of a $G(M)$-invariant Keisler measure over $M$ on $G$, whenever $M$ is a model over which $G$ is defined.  This  is in turn equivalent to the existence of a $G(M)$-invariant Borel probability measure on the space $S_{G}(M)$. In any case, note  that if all subsets of $G(M)$ happen to be definable, then definable amenability of $G$ is equivalent to amenability of $G(M)$ as a discrete group. 

\vspace{5mm}
\noindent
$T$ is said to be (or have) $NIP$ if for any indiscernible sequence $(b_{i}:i<\omega)$, formula $\phi(x,y)$ and $a$, there is an eventual truth value of $\phi(a,b)_{i})$ as $i\to \infty$. 
$o$-minimal theories are $NIP$.

Here are  a few  consequences of $NIP$.

\begin{Fact} Assume that $T$ is $NIP$, $M$ a small model,  and where appropriate that $G$ is a definable group defined over $M$. 
\newline
(i)  (\cite{NIPII})   A global type $p(x)\in S({\bar M})$ does not divide over a (small) model $M$ if and only if $p$ is $Aut({\bar M}/M)$-invariant. 
\newline
(ii) (\cite{NIPI})  Let $M_{0} = M^{ext}$. Then $Th(M_{0})$ has quantifier elimination and $NIP$. 
\newline
(iii) (\cite{NIPI})  $G^{00}_{M}$, the smallest type-definable over $M$ subgroup of $G$ of bounded index, does not depend on the choice of $M$, and we just write it as $G^{00}$. 
\newline
(iv) (\cite{CPS})   $G^{00}$ is the same whether computed in $T$ or in $Th(M^{ext})$.
\newline
(v) (\cite{NIPII})   $G$ is definably amenable iff there exists $p(x)\in S_{G}({\bar M})$ such that for every $g\in G = G({\bar M})$, $gp$ does not divide over $M$.  Following the notation of \cite{Chernikov-Simon} we call  a type $p$ as in the right hand side a (global) strongly $f$-generic, over $M$, type of $G$. 

\end{Fact}

The recent paper \cite{Chernikov-Simon} gives, among other things, a fairly comprehensive account of the relations betweeen different notions of genericity in the $NIP$ definably amenable context.  Although we will only be using ``easy" directions of their observations, there is no harm in describing some of their results.   Given a definable subset $X$ of $G$, they define $X$ to be $f$-generic if for some/any  model $M'$ over which $X$ is defined any left translate $gX$ of $X$ does not divide over $M'$ (i.e. for some/any model $M'$ over which $X$ is defined and any $M'$-indiscernible sequence $(g_{i}:i<\omega)$, $\{g_{i}X: i< \omega\}$ is inconsistent).  As the notation suggests, the property does not depend on the model $M'$ chosen.  Call a complete type $p$ (over some set of parameters) $f$-generic iff every formula in $p$ is $f$-generic.  The following appears in \cite{Chernikov-Simon}.

\begin{Fact}  ($T$ $NIP$)
\newline
(i) Assume $G$ definably amenable and let $X$ be a definable subset of $G$. Then $X$ is $f$-generic iff $X$ is (globally) weakly generic. In particular for $p\in S_{G}({\bar M})$, $p$ is $f$-generic iff $p$ is weakly generic.
\newline
(ii)   $p\in S_{G}({\bar M})$ is $f$-generic iff $Stab(p) = \{g\in G({\bar M}): gp = p\} = G^{00}$. 
\newline
(iii) Any global  strongly $f$-generic type is $f$-generic, but there are examples of $f$-generic types which are NOT strongly $f$-generic. 
\newline
(iv) $G$ is definably amenable iff there is a global $f$-generic type. 
\end{Fact}

The facts as stated above do not directly settle the question of whether (in our current $NIP$, definably amenable context) if $X$ is a globally weakly generic definable set and $X$ is defined over $M$ we can witness weak genericity with a formula over $M$. But this will follow, at least in the $o$-minimal case from our results in section 5.

For the rest of this section we assume that $T$ has $NIP$. We begin with some elementary observations on the weak generic = almost periodic question.  Fact 1.13(v) and Fact 1.14(iv) explicate definable amenability of $G$ in terms of (strongly) $f$-generic types. It is easy to see that a global   $f$-generic type $p$ is strongly $f$-generic just if $p$ does not divide over  (equivalently is invariant over) some small model  (because then by Fact 1.14 (ii), (iii),  there is a fixed small model over which every translate of $p$ is (automorphism)  invariant, which is enough by 5.11 of  \cite{NIPII}).  Hence there are two extreme cases for a strongly $f$-generic type $p$:
\newline 
(I) $p$ is {\em definable}, and we call $p$ a definable $f$-generic,
\newline
(II) $p$ is finitely satisfiable in some small model, and we call $p$ a finitely satisfiable $f$-generic. 

Noting that weak generic is the same as $f$-generic (Fact 1.14) we  give positive answers to the ``main question" in these two cases.

\begin{Lemma} Suppose that $p\in S_{G}({\bar M})$ is a definable $f$-generic. Then $p$ is almost periodic.
\end{Lemma} 
\begin{proof}  This is an adaptation of standard stability methods. By Fact 1.14,  $Stab(p) = G^{00}$. Let $\Delta$ be a finite set of $L$-formulas $\phi(x,y)$  ($x$ fixed of sort $G$ and $y$ varying).  Let $\Delta^{*}$ be the collection of definable subsets of $G$ of the form $g\phi(x,a)$ for $g\in G$ and $a$ in ${\bar M}$ and $\phi(x,y)\in \Delta$. Let $p|(\Delta^{*})$ be the collection of definable sets from $\Delta^{*}$ which are in $p$. By definability of $p$, $Stab_{\Delta}(p) = _{def} Stab(p|(\Delta^{*})) = \{g\in G({\bar M}):$ for every $X$ in $\Delta^{*}$, $X\in p$ iff $gX\in p\}$ is a definable subgroup of $G$ which contains $G^{00}$, hence, by saturation of ${\bar M}$  has finite index in $G$.  Now clearly $Stab(p)$ is the intersection of the $Stab_{\Delta}(p)$ as $\Delta$ ranges over all finite sets of formulas, so in fact $G^{00} = G^{0}$, the intersection of all definable subgroups of finite index. 

Note that for each finite $\Delta$, there are only finitely many $G({\bar M})$-translates of $p|(\Delta^{*})$, given by the cosets of $Stab_{\Delta}(p)$, and of course each of these is also definable.   We claim that the orbit $G({\bar M})\cdot p$ is already closed (so $p$ is almost periodic).  Let $q\in S_{G}({\bar M})$ with $q\in cl(G({\bar M})\cdot p)$.  It is easy to see that for each finite $\Delta$, $q|(\Delta^{*})$ is one of the finitely many translates of $p|(\Delta^{*})$. In particular $q$ is also a definable type. By compactness we find some $g\in G({\bar M})$ such that $gp = q$.

\end{proof}

If $G$ has some definable global $f$-generic it need not be the case that every global $f$-generic is definable. Nevertheless we would guess that:
\newline
{\em Conjecture.}  If $G$ has a definable $f$-generic then any global $f$-generic (weak generic) type is almost periodic.

\vspace{2mm}
\noindent
A related result is:
\begin{Lemma} Suppose that $G$ is definable over $M$ (where $M$ may be saturated) and $G = G^{00}$.  Then working in $S_{G}(M)$ or $S_{G}(M^{ext})$, any weak generic type is $G(M)$-invariant hence almost periodic. 
\end{Lemma}
\begin{proof} As $G^{00}$ is unchanged in $Th(M^{ext})$ it is enough to work with $S_{G}(M)$. Let $p\in S_{G}(M)$ be weakly generic. So every formula in $p$ is weak generic (with respect to $M$) so $f$-generic, so $p$ extends to a global $f$-generic type $p'$. Then $Stab(p') = G^{00} = G$. Hence $p$ is $G(M)$-invariant, as required. 
\end{proof} 

\vspace{2mm}
\noindent
On the other hand, if $G$ has a global $f$-generic type $p$ of kind (II), then this means precisely that $G$ is $fsg$ in the sense of \cite{NIPI}. In particular any global weak generic ($f$-generic) is generic. We summarize the situation in that case:

\begin{Remark} Suppose $G$ is $fsg$. Then for any model $M$ over which $G$ is defined, including the monster  model ${\bar M}$, in both $S_{G}(M)$ and $S_{G}(M^{ext})$ there is a unique minimal flow, the space of generic types. Hence weak generic = almost periodic in $S_{G}(M)$ or $S_{G}(M^{ext})$
\end{Remark}
{\em Explanation.}  This appears explicitly in Fact 5.5 of \cite{CPS} with references. But the point is just that there {\em exist}  generic types in $S_{G}(M)$ and $S_{G}(M^{ext})$, and then one can use Lemma 1.7 (ii). The existence of generic types in $S_{G}(M)$ is part of the general theory of $fsg$ groups (see \cite{NIPI} and \cite{NIPII}), and it is proved in Theorem 3.19 of \cite{CPS}, that $G$ is also $fsg$ in $Th(M^{ext})$.

\vspace{2mm}
\noindent
 It is also worth remarking here, that  the space of generic types in $S_{G}({\bar M})$ (where ${\bar M}$ is a saturated model of $T$) is homeomorphic to the space of generic types of $S_{G}(M^{ext})$, as  proved in \cite{Pillay-fsg}.  The (natural) homeomorphism $h$  is as follows: given generic type $p\in S_{G}({\bar M})$, let $h(p)$ be axiomatized by $X\cap G(M)$, for $X\in p$. So the point is that $h(p)$ is a complete type over $M^{ext}$, and also determines $p$. 
\newline
{\em Notation.}  Given $v$ a global generic type of $G$  (over the saturated model ${\bar M}$ of $T$), and small submodel $M$ of ${\bar M}$, let $v_{M}$ be the corresponding generic type in $S_{G}(M^{ext})$.

\subsection{$o$-minimality}
We now specialize to the  case where $T$ is $o$-minimal, namely a complete $o$-minimal expansion of the theory $RCF$ of real closed fields.  The reader is referred to \cite{CP} for relevant background on definable groups in $o$-minimal structures as well as further references. 

 If $G$ is a definable group, then $G^{0}$ has finite index, so we often assume $G = G^{0}$ in which case we call $G$ definably connected.  $G$ is equipped with its ``definable topology", and is said to be definably compact \cite{PS} if for any  definable continuous function $f:[0,1)$ to $G$, $lim_{x\to 1}f(x)$ exists and is in $G$. Here $[0,1)$ is the appropriate interval in the appropriate real closed field. 

We have the following decomposition theorem  \cite{CP}, Proposition 4.6:
\begin{Fact} Suppose $G$ is definably connected. Then $G$ is definably amenable iff $G$ has a definably connected torsion-free definable normal subgroup  $H$ such that $G/H$ is definably compact.
\end{Fact} 

In fact $H$ is ``characteristic" from the model-theoretic point of view, being the unique maximal normal definable torsion-free subgroup of $G$: it is defined over $A$ if $G$ is. And Proposition 2.1(i) of \cite{CP} says that $H$ is solvable and there are definable $1 = H_{0} < ... < H_{n} = H$ where each $H_{i+1}/H_{i}$ is $1$-dimensional, torsion-free (definably connected).  Also by Proposition 4.7 of \cite{CP} $H$ has a global left invariant definable ($f$-generic) type, so $H$ fits into one of the extreme cases in the previous section.

On the other hand from \cite{NIPI} we know that the definably compact groups are precisely the $fsg$ groups (in the $o$-minimal context), so they are the other extreme case.

Simon \cite{Simon-distal} has isolated a notion, distality, meant  to express the property that a $NIP$ theory $T$ has ``no stable part", or is ``purely unstable".  Examples include $o$-minimal and ``$p$-minimal" theories. 

\begin{Question} Is there a structure theorem for definably amenable groups $G$  in distal  theories along the lines of Fact 1.18. For example is $G$ an extension of $C$ by $H$ where $C$ is $fsg$ and $H$ has a definable global $f$-generic? 
\end{Question} 

\begin{Remark} Suppose $H$ is torsion-free (definably connected) then $H = H^{00}$.  Hence by Lemma 1.16, weak generic = almost periodic = $H(M)$-invariant, working in $S_{H}(M)$ or $S_{H}(M^{ext})$ for any $M$ over which $H$ is defined. 
\end{Remark}
\begin{proof} Because there is a global left invariant type. 
\end{proof} 

In Section 5.8 of \cite{CPS}, the ``Ellis group conjecture" is proved in the $o$-minimal situation and the analysis  given there will be useful for the current paper.  So we revisit Lemma 5.11 of \cite{CPS}.

Let us fix definably amenable, definably connected $G$, defined over $M$, and let $1\to H\to G\to C \to 1$ be the definable (over $M$) exact sequence from Fact 1.18.  Let $M_{0} = M^{ext}$. Then passing to $Th(M_{0})$, $C$ is still $fsg$ and $H= H^{00}$ and has a global definable left-invariant type.   We now work in $T' = Th(M_{0})$.  Let $M'$ denote a saturated elementary extension of $M_{0}$. Let $\pi:G\to C$. So $\pi$ induces a surjective continuous map which we still call $\pi$ from $S_{G}(M_{0})$ to $S_{C}(M_{0})$.

\begin{Lemma}  Let ${\cal M}(G)$ be a minimal subflow of $S_{G}(M_{0})$, and let ${\cal M}(C)$ be the unique minimal subflow of $S_{C}(M_{0})$. 
\newline
(i) For $p\in {\cal M}(G)$ and $g\in G(M)$, $g\cdot p$ depends only on the coset $g/H\in C$ of $g$, so we write it as $(g/H)\cdot p$, yielding an action of $C(M)$ on ${\cal M}(G)$ (by homeomorphisms). 
\newline
(ii)  $\pi$ induces a homeomorphism between ${\cal M}(G)$ and ${\cal M}(C)$  which respects the actions by $C(M)$, so is an isomorphism of $C(M)$-flows. 
\newline
(iii)  The map $(g/H),p) \to (g/H)\cdot p$ from $C(M)\times {\cal M}(G)$ to ${\cal M}(G)$, extends uniquely to a map $(q,p)\to q*p$ from $S_{C}(M_{0})\times {\cal M}(G)$ to ${\cal M}(G)$,  such that $\pi(q*p) = q*(\pi(p))$ where the $*$ on the right hand side is induced by the usual semigroup structure on $S_{C}(M_{0})$.
\newline
(iv)  For any $p\in {\cal M}(G)$, there is an idempotent $v\in {\cal M}(C)$ such that $v*p = p$ where $*$ is as in (iii). 
\newline
(v) For $p\in {\cal M}(G)$, ${\cal M}(G) = \{q*p: q\in {\cal M}(C)\}$. 
\end{Lemma}
\begin{proof}  Everything is contained in the proof of Lemma 5.11 in \cite{CPS} except possibly (iv).  But we first recall the definition of $q*p$ for $q\in S_{C}(M_{0})$ and $p\in {\cal M}(G)$ as it appears later.  Now as $p$ is definable and $H(M)$-invariant, the unique global heir $p'\in S_{G}(M')$ of $p$ is $H(M')$-invariant, so for $g\in G(M')$, $gp'$ depends only on the coset $g/H$, so we write it as $(g/H)\cdot p'$. For $q = tp((g/H)/M_{0})\in S_{C}(M_{0})$, define $q*p$ to be  $((g/H)\cdot p')|M_{0}$. 
\newline
Proof of (iv):  Let $p\in {\cal M}(G)$. Then $\pi(p)\in {\cal M}(C)$, so by Chapter I, Proposition 2.3 of \cite{Glasner} for example, there is an idempotent $v\in S_{C}(M_{0})$ such that $v*\pi(p) = \pi(p)$. By (iii), $v*p = p$.

\end{proof}

\section{Global case}
Here we prove Theorem 1.1. 
We assume from the beginning that $G$ is a definable, definably connected, definably amenable group, which is defined in the saturated $o$-mininmal structure $M$. 
We have the canonical exact sequence $H \to G \overset{\pi}{\rightarrow} C$ as described earlier.   As $T$ has Skolem functions, let $f:C\to G$ be a definable section of $\pi$ (but of course $f$ need not be a homomorphism, as $G$ is not necessarily split). 
We aim to show that if $dim(H) = 1$  then any global weakly generic type of $G$ is almost periodic. 

\begin{Lemma} If $dim(H)$ is $1$ then $H$ is central in $G$.
\end{Lemma} 
\begin{proof}  This should be considered folklore, and there are several ways of seeing it. Here is one. We consider the action of $C$ on $H$ (induced by conjugation in $G$). If by way of contradiction this action is not trivial, then the definable subgroup of $H$ consisting of the fixed points of $C$ has to be trivial as $dim(H) = 1$ and $H$ is definably connected. Every other $C$-orbit has to be infinite  so again as $dim(H) = 1$ there are only finitely many infinite orbits. As $C$ is definably compact, each of these orbits is closed, but then we contradict definable connectedness of $H$.

\end{proof}

\begin{Lemma} Suppose $H$ is central Let $p = tp(a/M)$ be a weak generic type of $G$.  Write $a$ (uniquely) as $f(\pi(a))\cdot h$ with $h\in H$. Then $\pi(p) = tp(\pi(a))/M)$ is a generic type of $C$ and $tp(h/M)$ is a left   $H(M)$-invariant type of $H$.
\end{Lemma} 
\begin{proof}  By 1.14,  $Stab(p) = G^{00}$.   So 
clearly $Stab(\pi(p)) = C^{00} = \pi(G^{00})$. and so is generic in $C$.  Note that $h = f(\pi(a))^{-1}\cdot a$.

Also by Remark 1.20, $H = H^{00}$ so is contained in  $G^{00}$ whereby $p$ is $H(M)$invariant. Let $h_{1}\in H(M)$.   Then $h_{1}\cdot h = h_{1}\cdot(f(\pi(a))^{-1}\cdot a) = (f(\pi(a))^{-1}\cdot h_{1}a$ as 
$H$ is central in $G$. But as $H(M)\subseteq Stab(p)$, $tp(h_{1}a/M) = tp(a/M)$.   Moreover as $\pi(h_{1}a) = \pi(a)$, $f(\pi(h_{1}a))^{-1} = (f(\pi(a))^{-1}$. It follows that
$tp(((f(\pi(a))^{-1}\cdot a/M) = 
tp((f(\pi(h_{1}a))^{-1})\cdot h_{1}a/M) = tp((f(\pi(a))^{-1}h_{1}a/M)  = tp(h_{1}h/M)$..
Hence $tp(h/M) = tp(h_{1}h/M)$ as required. 
\end{proof}

Let us now assume $dim(H) = 1$. Then there are two $H(M)$-invariant types in $S_{H}(M)$, $p_{0}^{+}$ and $p_{0}^{-}$, at $+$-infinity and $-$-infinity respectively, and they are both definable. Let $p = tp(a/M)$ be our global weakly generic type of $G$, and write  as above $a = (f(\pi(a))\cdot h$.  So $tp(h/M)$ is global $H(M)$-invariant type of $H$ and we know it is $p_{0}^{+}$ or $p_{0}^{-}$ Without loss it is $p_{0}^{+}$. 
\begin{Remark} $p_{0}^{+}$ is weakly orthogonal to $q$ (as well as $f(q)$) for any global generic type $q$ of $C$.

\end{Remark}
\begin{proof} Remember that this weak orthogonality means that $p_{0}^{+}(x_{0})\cup q(x_{1})$ extends to a unique complete global type in $(x_{0},x_{1})$.

Now $p_{0}^{+}$ can be assumed to be a complete $1$-type over $M$.  If $p_{0}^{+}$ is not weakly orthogonal to $q$ then by $o$-minimality, some realization $h$ of $p_{0}^{+}$ is in the definable closure, over $M$ of $b$ for some realization $b$ of $q$.
But then there is a small elementary substructure $N$ of $M$ such that $tp(h/M)$ is both definable over and finitely satiasfiable in $N$. This is easily seen to be impossible, by inspection.

\end{proof} 

Hence for $q$ a global generic type of $C$, we can speak about the type $f(q)\cdot p_{0}^{+}$, which is defined to be  $tp(f(b)\cdot h/M)$ where $b$ realizes $q$ and $h$ realizes $p_{0}^{+}$.  By weak orthogonality it is unique, and in particular $tp(h/M,b)$ is the unique heir of $p_{0}^{+}$ over $M,b$.  Again we fix realization $a$ of the global weak generic type $p$, and write $a = (f(\pi(a))\cdot h$. 

\begin{Lemma} $cl(G(M)\cdot p)$ is precisely the set of types $f(q)\cdot p_{0}^{+}$ as $q$ ranges over global generic types of $C$. 

\end{Lemma} 
\begin{proof} We know that $cl(G(M)\cdot p) = \{tp(c\cdot a/M): c\in G$ and $tp(c/M,a)$ is finitely satisfiable in $M$\}.  Let $c$ be such that  $tp(c/M,a)$ is finitely satisfiable in $M$. Then $tp(c/M,a,h)$ is finitely satisfiable in $M$. Now $\pi(c\cdot f(\pi(a)) = \pi(c)\pi(a)$. Also $\pi(c)\in C$ and $tp(\pi(c)/M,\pi(a))$ is finitely satisfiable in $M$. Hence as the space of generic types of $C$ is closed and $C(M)$-invariant, it follows that  $\pi(c)\pi(a) = d$ realizes a generic type of $C$.
So the conclusion is that $\pi(c\cdot f(\pi(a))) = d$  realizes a generic type $q$ say of $C$. 
\newline
It follows that $c\cdot f(\pi(a)) = f(d)\cdot h_{1}$ for some $h_{1}\in H$. 

Now as $tp(a/M,h)$ is finitely satisfiable in $M$, and $h_{1}\in dcl(M,c,a)$ it follows that $tp(h_{1}/M,h)$ is finitely satisfiable in $M$, whereby $h$ realizes $p_{0}^{+}|(M,h_{1})$  (the unique heir of $p_{0}^{+}$) and so $h_{1}h$ realizes $p_{0}^{+}$ too.
So the end result is that $ca = cf(\pi(a))h = f(d)h_{1}$ realizes $f(q)\cdot p_{0}^{+}$ as required.

\end{proof}

\begin{Corollary} $cl(G(M)\cdot p)$ is a minimal $G(M)$-flow, in particular $p$ is almost periodic. 
\end{Corollary}
\begin{proof} Let $X = cl(G(M)\cdot p)$ and let $Y$ be the space of global generic types of $C$ (which we know to be the unique minimal $C(M)$- subflow of $S_{C}(M)$).  So $\pi$ induces a surjective continuous map $\pi|X: X \to Y$. By Lemma 2.4,  $\pi$ is one-to-one, hence is a homeomorphism.  Now for $g\in G(M)$, and $r\in X$, $\pi(gr) = \pi(g)\pi(r)$, so as $Y$ is a minimal $C(M)$-flow and $\pi|X$ is $1-1$ it follows that $X$ is a minimal $G(M)$-flow.

\end{proof}

So we have shown that if $dim(H) = 1$ then any global weakly generic type of $G$ is almost periodic.  In fact we have shown that there are precisely two minimal subflows of $S_{G}(M)$ corresponding to $p_{0}^{+}$ and $p_{0}^{-}$. 

In any case, using also 1.17 and 1.20 we have proved Theorem 1.1.

\section{Local case}
We aim here towards the proof of Theorem 1.2, and again it is enough to prove the $dim(H) = 1$ case by 1.17 and 1.20.   
 We can no longer use the weak orthogonality arguments, but on the other hand we have definability of types, as well as the semigroup structure $*$ on the type space $S_{G}(M_{0})$. The additional assumption that $G$ is a product is needed for technical reasons, and we leave it as a problem to get around it. 

So we have $M$ an arbitrary model of $T$ over which $G$ is defined, and $M_{0} = M^{ext}$.  $\bar{M_{0}}$ denotes a saturated elementary extension of $M_{0}$, where we can realize types over $M_{0}$. 

We will for now take $G$ to be an arbitrary definably connected definably amenable group, definable over $M$, and we again have

$$ 1\to H \to G \overset{\pi}{\rightarrow} C \to 1$$

As in the beginning of the proof of Lemma 1.21 we have:

\begin{Remark} Let $p\in S_{G}(M_{0})$ be $H(M)$-invariant. Then for $r= tp(g/M_{0})$, $r*p$ depends only on the coset $gH$ of $g$, so we can write $r*p$ as $\pi(r)*p$.
\end{Remark}

The key result is the following: 
\begin{Lemma} Let $p\in S_{G}(M_{0}$. Then $p$ is almost  periodic iff $p$ is $H(M)$-invariant, and $v*p = p$ for some generic type $v\in S_{C}(M_{0})$ of $C$ (where in fact $v$ can be chosen to be an idempotent). 
\end{Lemma}
\begin{proof}  Left to right is contained in Lemma 1.21.  The ``new" observation here is the converse.

Note that $cl(G(M_{0})\cdot p) = S_{G}(M_{0})*p$. So in fact $p$ is almost periodic iff 
\newline
(**)  for each $q\in S_{G}(M_{0})$, $S_{G}(M_{0})*q*p = S_{G}(M_{0})*p$.
\newline
Let us assume the right hand side in the Lemma.  We have to prove (**).  Fix $q\in S_{G}(M_{0})$, so we have to prove that
\newline
$(S_{G}(M_{0})*q)*p = S_{G}(M_{0})*p$.
\newline
As $p$ is $H(M)$-invariant, by Remark 3,1 we have to prove:
\newline
(***)  $(S_{C}(M_{0})*\pi(q))*p = S_{C}(M_{0})*p$. 
\newline
Now we are supposing that $v*p = p$ for some generic (so almost periodic) type $v\in S_{C}(M_{0})$.
So the right hand side in (***) coincides with $(S_{C}(M_{0})*v)*p$.   Similarly the left hand side of (***) coincides with $(S_{C}(M_{0})*\pi(q)*v)*p$.
But applying (**) to $C$ and using the almost periodicity of $v$ we see that $(S_{C}(M_{0})*\pi(q)*v) = S_{C}(M_{0})*v$.  This proves (***).

\end{proof}

We now apply this lemma to the case where $G$ is a product.

\begin{Lemma}  Suppose in addition that $G = C\times H$. Let $p = tp(a/M) \in S_{G}(M_{0})$. Suppose that $a = (c,h)$ for $c\in C$ and $h\in H$. 
Then $p$ is almost periodic if and only if $tp(c/M_{0})$ is generic for $C$ (equivalently almost periodic for $C$), $tp(h/M_{0})$ is $H(M_{0})$-invariant 
(equuivalently almost periodic for $H$) and $h$ realizes the unique heir of $tp(h/M_{0})$ over $M_{0},c$. 

\end{Lemma}
\begin{proof}  If $p$ is as in the right hand side then it is clearly $H(M)$-invariant. Let $v$ be an almost periodic type of $C$ such that 
$v*tp(c/M_{0}) = tp(c/M_{0})$  (as referenced in the proof of 1.21 such $v$ exists). Then  for $d$ realizing $v$ such that $tp(d/M_{0},c,h)$ is finitely 
satisfiable in $M_{0}$, $tp(dc/M,h)$ is finitely satisfiable in $M_{0}$ so $tp(dc,h/M_{0}) = tp(c,h/M_{0}) = p$, so  $v*p = p$. So by Lemma 3.2, $p$ is almost periodic.

Conversely, if $p = tp(c,h/M_{0})$ is almost periodic, then it satisfies the conditions of Lemma 3.2. So  clearly $tp(h/M_{0})$ is $H(M)$-invariant. Now suppose $v*p = p$ for some generic $v\in S_{C}(M_{0})$. Let $v'$ be the unique coheir of $v$ over $\bar{M_0}$. So $v'$ is a global (for $Th(M_{0})$) (left and right) generic type of the definably compact group $C$, finitely satisfiable in $M_{0}$. Let $d$ realize $v'$. So $dc$ is also a global generic type of $C$, so finitely satisfiable also in $M_{0}$.  But as $v*p = p$, $tp(dc,h/M_{0}) = tp(c,h/M_{0})$ which means that $c$ realizes a generic type of $C$ and its type over $M_{0},h$ is finitely satisfiable in $M_{0}$. This proves that $p$ satisfies the right hand side conditions of the lemma.
\end{proof}

\begin{Lemma} Suppose $dim(H) = 1$ (as a definable group in the original $o$-minimal theory). Then  there are only two $H(M_{0})$ invariant types in $S_{H}(M_{0})$ (equivalently only two almost periodic types). 
\end{Lemma}
\begin{proof}  Suppose $p\in S_{H}(M_{0})$ is $H(M)$-invariant. Then so is the restriction $p|M$ to $S_{H}(M)$. But we know that $S_{H}(M)$ has only two $H(M)$-invariant types $p_{0}^{+}$ and $p_{0}^{-}$, the types at plus and minus infinity. Each of these types is definable, so by Lemma 3.18 of \cite{CPS},  have unique completions in $S_{H}(M_{0})$, each of which of course remains $H(M)$-invariant.  
\end{proof}

\vspace{2mm}
\noindent
{\em Completion of proof of Theorem 1.2  in the case where $dim(H) = 1$. }
\newline
 Let $p_{0}^{+}$, $p_{0}^{-}$ denote the two $H(M)$-invariant types in $S_{H}(M_{0})$ (given by the previous lemma).  Let $Y$ be the closed set of generic types of $C$ over $M_{0}$. Then by Lemma 3.3, the set of almost periodic types of $G$ over $M_{0}$ is the set of $tp(c,h/M_{0})$ where $c$ realizes a type in $Y$ and $h$ realizes the unique heir of either $p_{0}^{+}$ or $p_{0}^{-}$ over $M_{0},c$. This is easily seen to be a closed set of types (using for example definability of $p_{0}^{+}$ and $p_{0}^{-}$). 

We have proved here that working in $S_{G}(M_{0})$, the set of almost periodic types is closed.  This proves Theorem 1.2, except for the  final sentence which will be deduced in section 5.

\section{The example} 
Here we prove Theorem 1.3.  We will again deal with the space $S_{G}(\R)$, and deduce the global case in Section 5. 

$M$ will be the field $\R$ of real numbers, and $G$ the product of the circle group $C$ by $H = (\R,+)^{2}$. We also use $G, C, H$ to refer to their  interpretations in a saturated elementary extension $\bar M$ of $M$. 
$C =\{(x,y): x^{2} + y^{2} = 1\}$ and we tend to identify a point of $C$ with its $x$-coordinate, working above the $x$-axis.

We aim to describe a type in $S_{G}(M)$ which is weakly generic but not almost periodic. 

So here the ambient complete theory is $RCF$. We will be using repeatedly the fact that
\newline
(*)  if $K$ is a model and $a>K$ then the set of positive powers $\{a^{n}:n>0\}$ is cofinal in $dcl(K,a)$. 

\vspace{2mm}
\noindent
As all types over $M$ are definable, we are actually already  in the local situation of the last section.

Let $a>\R$, $b_{1}> dcl(\R,a)$ and $b_{2} =  b_{1}a$.    Consider the cut in  $dcl(\R(a,b_{1}))$ whose left hand part is $\{e\in dcl(\R(a,b_{1})): e \leq$ some element of $dcl(\R(a))\}$. 
Let $q$ be the corresponding complete $1$-type over $dcl(\R(a,b_{1}))$, and let $c$ realize $q$. Finally let $d= c^{-1}$.  Noting that $0< d< a^{-1}$, and that $a^{-1}$ is infinitesimal, we see that:
\newline
(I) $tp(d/M, b_{1}, b_{2})$ is not finitely satisfiable in $M$. 

\vspace{5mm}
\noindent
As mentioned above, we view $d$ as an element of $C$ by identifying it with $(d, +\sqrt{1-d^{2}})$.  Hence by Lemma 3.3  we have:
\begin{Lemma}  $p = tp((d,b_{1},b_{2})/M)$ is not an almost periodic type of $G$.
\end{Lemma}

On the other hand we will show:
\begin{Lemma}  $p = tp((d,b_{1},b_{2})/M)$ is a weakly generic type of $G$.
\end{Lemma}
\begin{proof}
We first want to give an  axiomatization of $p$. First $tp(b_{1}/M)$ is axiomatized by $\{y_{1}> r: r\in M\}$, i.e. $\{y_{1}>n: n\in \N\}$.  
\newline
{\em Claim I.}  $tp(b_{2}/M,b_{1})$ is axiomatized by $\{y_{2} > nb_{1}: n\in \N\} \cup \{y_{2} < b_{1}^{1 + 1/n}:n\in \N\}$. 
\newline
{\em Proof.}  Note that using (*),  $tp(a/M,b_{1})$ is axiomatized by $\{x> n:n\in \N\} \cup \{x<b_{1}^{1/n}: n\in \N\}$. As $b_{2} = ab_{1}$, we obtain the claim. 

\vspace{2mm}
\noindent
{\em Claim II.}  $tp(d/M,b_{1},b_{2})$ is axiomatized by   $\{x^{-1}> (b_{2}/b_{1})^{n}: n\in \N\}\cup \{x^{-1} < b_{1}^{1/n}: n\in \N\}$. 
\newline
{\em Proof.}  This is also straightforward but we give some details. The statement is equivalent to saying that $tp(c/M,a,b_{1})$ is axiomatized by $\{x> a^{n}:n\in \N\} \cup \{x<b_{1}^{1/n}: n\in \N\}$.  Clearly $c$ satisfies this set of formulas. If the set of formulas does not determine a complete type over $M,a,b_{1}$ there is some $M$-definable function $g$   such that $a^{n} < g(a,b_{1}) < b_{1}^{1/n}$ for all $n\in \N$. 
Now as $b_{1}>dcl(M,a)$ and $g(a,b_{1})> dcl(M,a)$, $g(a,y)$ has to be an order-preserving isomorphism between $(h_{1}(a),+\infty)$ and $(h_{2}(a), +\infty)$ for some $h_{1}, h_{2}$ over $M$.  So the compositional inverse of $g(a,-)$ is an order preserving isomorphism between $(h_{2}(a),+\infty)$ and $(h_{1}(a),+\infty)$. By (*) above, $b_{1} < g(a,b_{1})^{n}$ for some $n$, whereby $b_{1}^{1/n} < g(a,b_{1})$, contradiction.  End of proof of Claim II. 

\vspace{3mm}
\noindent
We continue with the proof of the lemma. 
We want to show that any formula $\phi(x,y_{1},y_{2})$  in $p$ is weakly generic, equivalently is in some almost periodic type over $M$. By the Claims  above we may assume that $\phi(x,y_{1},y_{2})$ is
$$(y_{1} > n_{0}) \wedge (ny_{1} < y_{2} < y_{1}^{1+1/n}) \wedge (   y_{1}^{-1/n}  <   x< (y_{2}/y_{1})^{-n}) $$

\noindent
where $n_{0}, n$ are positive natural numbers. Note that increasing $n_{0}$ and/or $n$ gives us another,  stronger,  formula which is also in $p$. 
So (increasing $n_{0}$) we may assume that if $y_{1} > n_{0}$, then $(n+1)y_{1} \leq y_{1}^{1+1/n}$. 

\noindent
Note that for  $ny_{1} < y_{2} < (n+1)y_{1}$,   $y_{2}/y_{1} < n+1$, hence
\newline
(i)  if $y_{1}>n_{0}$ and $ny_{1} < y_{2} < (n+1)y_{1}$, then $(ny_{1} < y_{2} < y_{1}^{1+1/n})$ and $(y_{2}/y_{1})^{-n} > (n+1)^{-n}$.

Also
\newline
(ii) $y_{1} > n_{0}$ implies $y_{1}^{1/n} > n_{0}^{1/n}$ implies $y_{1}^{-1/n} < n_{0}^{-1/n}$. 

\noindent
Enlarge $n_{0}$ if necessary so that 
\newline
(iii) $n_{0}^{-1/n} < (n+1)^{-n}$.

Hence by (i), (ii), and (iii),  if $y_{1} > n_{0}$, and $ny_{1} < y_{2} < (n+1)y_{1}$ and $n_{0}^{-1/n} < x < (n+1)^{-n}$, then 
$\phi(x,y_{1},y_{2})$ holds. 

But  the formula  $(y_{1} > n_{0})\wedge (ny_{1} < y_{2} < (n+1)y_{1})$ is weakly generic in $H$, because it is in some $H(M)$- invariant type, such as $tp(c,d/M)$ where $c> M$ and $d$ realizes the  type corresponding to the cut whose left hand side is  $\{c+r:r\in\R\}$. 

And also the formula  $n_{0}^{-1/n} < x < (n+1)^{-n}$ is generic in $C$, because in fact any infinite definable over $M$ subset of $C$ is generic. 

So by Lemma 3.3 we find an almost periodic type of $G$ containing the formula $y_{1} > n_{0}$ and $ny_{1} < y_{2} < (n+1)y_{1}$ and $n_{0}^{-1/n} < x < (n+1)^{-n}$, so it also contains $\phi(x,y_{1},y_{2})$ which is then weakly generic.  This concudes the proof of the lemms and also of Theorem 1.3,  at least insofar as $S_{G}(\R)$ is concerned. 

\end{proof}

\section{Restrictions and extensions of almost periodic types}

We prove here Theorem 1.4, and use some of the additional observations to complete proofs of Theorems 1.2 and 1.3. 

We begin in the general context, namely $T$ an arbitrary theory, $G$ a group definable over $M$. We will also assume $T$ is $NIP$ so that  $Th(M^{ext})$ has quantifier eliimination and $S_{G,ext}(M) = S_{G}(M^{ext})$, but it is not really necessary. 

\begin{Lemma}  Suppose $p\in S_{G}(M)$. Then $p$ is almost periodic if and only if  $p$ has an almost periodic extension $p'$ in $S_{G}(M^{ext})$. 
\end{Lemma} 
\begin{proof} Left to right: Let $p'\in S_{G}(M^{ext})$ be an arbitrary extension of $p$. Let ${\cal M}\subseteq S_{G}(M^{ext})$ be a minimal subflow of $cl(G(M)\cdot p')$. Then the projection of ${\cal M}$ to $S_{G}(M)$ is a subflow of $S_{G}(M)$ contained in $cl(G(M)\cdot p)$, so has to equal $cl(G(M)\cdot p)$ as the latter is minimal. So any preimage of $p$ in ${\cal M}$ is an almost periodic point of $S_{G}(M^{ext})$. 
\newline
Right to left: Suppose $p'\in S_{G}(M^{ext})$ is almost periodic and let ${\cal M} = cl(G(M)\cdot p')\subseteq S_{G}(M^{ext})$ be the minimal subflow it generates. Then easily the projection of ${\cal M}$ on $S_{G}(M)$ is also a minimal subflow, so $p$ is almost periodic.
\end{proof} 

\begin{Lemma} Suppose $p\in S_{G}(M)$ is weak generic. Then $p$ has an extension $p'\in S_{G}(M^{ext})$ which is weak generic.
\end{Lemma} 
\begin{proof}  By Remark 1.8 (ii) a definable subset $Y$ of $G(M)$ is weak generic (in $Th(M)$) iff it is weak generic in $Th(M^{ext})$.  As the collection of non weak generic definable subsets of $G(M)$ in $Th(M^{ext})$ is a proper ideal, $p$ extends to a weak generic type in $S_{G}(M^{ext})$.
\end{proof} 

So by Lemmas 5.1 and 5.2 we can conclude the final sentence in Theorem 1.2: Suppose $p\in S_{G}(M)$ is weak generic. Extend to weak generic $p'\in S_{G}(M^{ext})$. By the first part of Theorem 1.2, $p'$ is almost periodic. By Lemma 5.1, so is $p$. 

\vspace{2mm}
\noindent
We now assume that $T$ is $o$-minimal, and that $G$ is definably amenable, definably connected, and defined over $M$. Let $1\to H \to G \to C \to 1$ be the canonical decomposition. 

\begin{Lemma} Let $M\prec N$, and $q\in S_{G}(N)$ be almost periodic. Then $p = q|M \in S_{G}(M)$ is almost periodic.
\end{Lemma} 
\begin{proof} We will show that for any $\phi(x)\in p$, $d_{p,M}(\phi)$ (an externally definable subset of $G(M)$) is left generic, and use  1.11(v). Now $q$ is $H(N)$-invariant, whereby $p$ is $H(M)$-invariant.  So for $g\in G(M)$, whether or not $\phi\in gp$ depends on $g/H\in C(M)$.  Hence $d_{p,M}(\phi)\subseteq G(M)$ is left generic iff $\pi(d_{p,M}(\phi))$ an externally definable subset of $C(M)$, is left generic. 

Now as $q$ {\em is} almost periodic, and $\phi\in q$, then $\pi(d_{q,N}(\phi))\in C(N)$ is left generic. By the discussion at the end of subsection 1.3 (or see \cite{Pillay-fsg}, $\pi(d_{q,N}(\phi))\cap C(M)$ is left generic. But by 1.11 (iv) $\pi(d_{q,N}(\phi))\cap C(M)$ is precisely $\pi(d_{p,M}(\phi))$. Hence $p$ is almost periodic, as required. 
\end{proof}

At this point we can also complete the proof of Theorem 1.3: We are given $p\in S_{G}(\R)$ which is weakly generic but not almost periodic. Let $p'\in S_{G}({\bar M})$ be a weakly generic type extending $p$. Then By Lemma 5.3, $p'$ is not almost periodic.

\begin{Proposition} Let $\cal M$ be a minimal subflow of $S_{G}(N)$  (as a $G(N)$-flow). Then ${\cal M}|S_{G}(M) = \{q|M: q\in {\cal M}\}$ is a minimal subflow of $S_{G}(M)$ (as a $G(M)$-flow).
\end{Proposition}

\begin{proof}  So ${\cal M} = cl(G(N)\cdot q)$ for an almost periodic $q\in S_{G}(N)$. 
By Lemma 5.3,  $p = q|M\in S_{G}(M)$ is almost periodic.   
By Lemma 5.1 let  ${\bar q}\in S_{G}(N^{ext})$ be an almost periodic extension of $q$.  So by 1.21 (iv)  there is a generic idempotent
$u\in S_{C}(N^{ext})$ such that  $u*{\bar q} ={\bar q}$.  By remarks at the end of section 1.3, we can write $u$ as $v_{N}$ for some global generic idempotent $v\in S_{C}({\bar M})$. 

Let $p'\in S_{G}(M^{ext})$ be an almost periodic extension of $p$ (by Lemma 5.1), and let  
${\bar p} =  v_{M}* p'$. 

\vspace{2mm}
\noindent
{\em Claim.}  ${\bar p}\in S_{G}(M^{ext})$ is almost periodic and extends $p$. 
\newline
{\em Proof of claim.}  Clearly ${\bar p}$ is almost periodic (by 3.2 for example).   We show it extends $p$.   Suppose $\phi(x)$ is a formula over $M$ and $\phi(x)\in p$. Then $\phi(x)\in {\bar q}$.  So $\phi(x)\in v_{N}*{\bar q}$. By Lemma 1.12, $\pi(d_{\bar q,N}(\phi)(N)))\in v_{N}$. So $\pi(d_{{\bar q},N}(\phi(N)))\cap C(M)\in v_{M}$. But by Remark 1.11, 
$d_{{\bar q},N}(\phi(N)) = d_{q,N}(\phi(N)) = d_{p,M}(\phi(M))\cap G(M)$ and the same after taking $\pi$ (i.e. projecting to $C$).  Hence $d_{p',M}(\phi)\in v_{M}$, so by Lemma 1.12, $\phi \in {\bar p} = v_{M}*p'$ as required. 

\vspace{2mm}
\noindent
Remember that $v$ and so  $v_{M}$, $v_{N}$ are idempotents. 

Now we have  the following almost periodic types: $q\in S_{G}(N)$, its restriction $p\in S_{G}(M)$, as well as ${\bar q}\in S_{G}(N^{ext})$ and ${\bar p}\in S_{G}(M^{ext})$. 
Let ${\cal M}_{q}$, ${\cal M}_{p}$, ${\cal M}_{\bar q}$, ${\cal M}_{\bar p}$, be the minimal subflows of the relevant spaces generated by the respective types.

Let $r:S_{G}(N)\to S_{G}(M)$, $r_{N}:S_{G}(N^{ext})\to S_{G}(N)$, $r_{M}:S_{G}(M^{ext})\to S_{G}(M)$ be the restriction maps.   So $r_{N}$ takes ${\cal M}_{\bar q}$ onto ${\cal 
M}_{q}$, and $r_{M}$ takes ${\cal M}_{\bar p}$ onto ${\cal M}_{p}$. 

We want to prove that $r$ takes ${\cal M}_{p}$ onto ${\cal M}_{q}$.   We will make use of  the following map $\chi: {\cal M}_{\bar p} \to {\cal M}_{\bar q}$, defined by 
$\chi(u_{N}*{\bar p}) = (u_{M}*{\bar q})$ for each global generic type $u$ of $C$.  This can be shown to be well-defined and onto using 1.21 (v).  

\vspace{2mm}
\noindent
To complete the proof of the Proposition it suffices to prove: 
\newline
{\em Claim.}  For $s\in {\cal M}_{\bar q}$, $(r_{M}\circ \chi)(s) = (r\circ r_{N})(s)$.
\newline
{\em Proof.}  Suppose $s = u_{N}*{\bar q}$ and $\phi(x)$ is an $L_{M}$-formula. Then $\phi(x)\in (r\circ r_{N})(u_{N}*{\bar q})$ iff  $\phi(x)\in u_{N}*{\bar q}$ iff 
$\pi(d_{{\bar q},N}(\phi(N)))  \in u_{N}$. So $\pi(d_{p,M}(\phi(M)))\in u_{M}$, whereby $\phi\in u_{M}*{\bar p}$, proving the claim.


\end{proof}


\begin{thebibliography}{99}



\bibitem{Auslander} J.Auslander, {\em Minimal flows and their extensions}, North Holland, Amsterdam, 1988.

\bibitem{Chernikov-Simon} A. Chernikov and P. Simon, Definably amenable NIP groups, preprint 2015.

\bibitem{CPS} A. Chernikov, A. Pillay, P. Simon, External definability and groups in NIP theories, Journal LMS, 90(1) (2014), 213-240.

\bibitem{CP}A. Conversano, A. Pillay, Connected components of definable groups and o-minimality I, Advances in Mathematics, vol.231(2012), pp. 605-623.



\bibitem{Glasner} E. Glasner, {\em Proximal flows}, Lecture Notes in Math., 517, Springer, 1976. 

\bibitem{GPPII} J. Gismatullin, D. Penazzi and  A. Pillay, On compactifications and the topological dynamics of definable groups, Annals of  Pure and Applied Logic, 165(2014), 552-562.

\bibitem{NIPI} E. Hrushovski, Y. Peterzil, and A. Pillay, Groups, measures and the $NIP$, Journal AMS 21 (2008), 563-596.

\bibitem{NIPII} E. Hrushovski and A. Pillay, On NIP and invariant measures, J. European Math. Soc. 13 (2011), 1005 --- 1061.







\bibitem{Marker-Steinhorn} D. Marker and C. Steinhorn, Definable types in $o$-minimal theories, Journal of Symbolic Logic 59 (1994), 185-198.




\bibitem{Newelski1} L. Newelski, Topological dynamics of definable group actions, Journal of Symbolic Logic, 74 (2009), 50-72.

\bibitem{Newelski-Petrykowski} L. Newelski and M. Petrykowski,  Weak generic types and coverings of groups I, Fund. Math., 191 (3) (2006), 201-225.



\bibitem{PS}Y. Peterzil, C. Steinhorn, Definable compactness and definable subgroups of $o$-minimal groups, J. London Math. Soc 59(1999), 769-786.






\bibitem{Pillay-fsg}  A. Pillay, Topological dynamics and definable groups, J. Symbolic Logic, 78:657-666,2013.






\bibitem{Poizat} B. Poizat, {\em A Course in Model Theory}, Springer, 2000. 




\bibitem{Simon-distal} P. Simon, Distal and non-distal NIP theories,  Annals of  Pure and Applied Logic,  164 (2013).


\end{thebibliography}
\end{document}